\documentclass[a4paper,reqno,11pt,oneside]{amsart}
\usepackage{graphicx}
\usepackage{xcolor}
\usepackage{amssymb}
\usepackage{amstext}
\usepackage{amsmath}
\usepackage{amscd}
\usepackage{amsthm}
\usepackage{amsfonts}
\usepackage{enumerate}
\usepackage{caption}
\usepackage{here}
\usepackage{bm}
\usepackage{tikz}
\usepackage{calc}
\usepackage{multirow}
\usepackage{makecell}
\usepackage[colorinlistoftodos]{todonotes}
\usepackage[hidelinks]{hyperref}
\usetikzlibrary{decorations.markings}
\usetikzlibrary{positioning}
\tikzstyle{vertex}=[circle ,draw, inner sep=0pt, minimum size=6pt]

\usepackage{comment}
\usepackage[whole]{bxcjkjatype}

\makeatletter
  
  \@addtoreset{equation}{section}
\makeatother

\newcommand{\Sc}{\mathcal{S}}

\newcommand{\ZZ}{\mathbb{Z}}

\newcommand{\RR}{\mathbb{R}}

\def\opn#1#2{\def#1{\operatorname{#2}}} 
\opn\conv{conv} \opn\dep{depth} \opn\Spec{Spec} \opn\cone{cone} \opn\ini{in} \opn\codeg{codeg} \opn\deg{deg}
\opn\Graph{Graph} \opn\sign{sign} \opn\Ehr{Ehr} \opn\rank{rank} \opn\type{type} \opn\reg{reg} \opn\core{core}
\opn\Hilb{Hilb} \opn\Indeg{Indeg} \opn\link{link} \opn\Tor{Tor} \opn\MNF{MNF} \opn\Stab{Stab} \opn\Vol{Vol}

\newtheorem{thm}{Theorem}[section]
\newtheorem{cor}[thm]{Corollary}
\newtheorem{lem}[thm]{Lemma}

\theoremstyle{definition}

\newtheorem{ex}[thm]{Example}

\theoremstyle{remark}

\begin{document}

\title{Ehrhart polynomials of cyclic polytopes as averages of zonotope Ehrhart polynomials}
\author{Masato Konoike}

\address{Department of Pure and Applied Mathematics, Graduate School of Information Science and Technology, The University of Osaka, Suita, Osaka 565-0871, Japan}
\email{kounoike-m@ist.osaka-u.ac.jp}


\begin{abstract}
We prove an averaging formula for the Ehrhart polynomial of a cyclic polytope whose vertices are given by integer parameters on the moment curve. More precisely, its Ehrhart polynomial is the average of the Ehrhart polynomials of an explicitly constructed family of lattice zonotopes. Since lattice zonotopes are magic positive, this formula implies magic positivity for these cyclic polytopes. Consequently, their $h^\ast$-polynomials are real-rooted, and their $h^\ast$-vectors are log-concave and unimodal.
\end{abstract}

\maketitle

\section{Introduction}\label{sec:introduction}

Let $P\subset\RR^d$ be a $d$-dimensional lattice polytope. Its Ehrhart polynomial $E_P(m)$ is characterized by
\[
E_P(m)=|mP\cap\ZZ^d| \qquad (m\in\ZZ_{>0}),
\]
and its $h^\ast$-polynomial is defined by
\[
1 + \sum_{m\ge1}E_P(m)z^m=\frac{h_P^\ast(z)}{(1-z)^{d+1}}.
\]
Following Ferroni and Higashitani~\cite{ferroni2024examples}, a polynomial $f$ of degree $d$ is called \emph{magic positive} if
\begin{equation}\label{eq:magic-basis} f(x)=\sum_{i=0}^d a_i x^i(1+x)^{d-i},\qquad a_i \ge 0. 
\end{equation}
The coefficients $a_i$ are called the \emph{magic coefficients} of $f$.

A theorem of Br\"anden implies, in the Ehrhart setting, that magic positivity of $E_P$ implies the real-rootedness of $h_P^\ast$; see~\cite{branden2006linear} and \cite[Theorem~4.19]{ferroni2024examples}. Since $h_P^\ast$ has nonnegative coefficients, its coefficient sequence is then log-concave and unimodal.

Magic positivity is known for several families of lattice polytopes; see, e.g., \cite{athanasiadis2026lattice,avila2026luck,beck2019h,ferroni2024examples,hill2026lattice,konoike2024new,konoike2025magic,liu2026magic}. For cyclic polytopes in the integral moment-curve realization, Liu~\cite{liu2005ehrhart} proved an explicit formula for their Ehrhart coefficients. In particular, these cyclic polytopes are Ehrhart positive.

We prove the stronger statement that the Ehrhart polynomial of $C_d(T)$ is the average of the Ehrhart polynomials of an explicitly constructed family of lattice zonotopes. Since lattice zonotopes are magic positive, this averaging identity immediately implies magic positivity for $C_d(T)$.

We consider the moment-curve realization
\[
\nu_d(t)=(t,t^2,\ldots,t^d),\qquad C_d(T)=\conv\{\nu_d(t):t\in T\},
\]
where $T=\{t_0<\cdots<t_n\}\subset\ZZ$ and $1\le d\le n$. We write $\Vol_j$ for $j$-dimensional Euclidean volume. Thus, if $p_0,\ldots,p_j\in\RR^j$ are affinely independent, then
\[
\Vol_j(\conv\{p_0,\ldots,p_j\}) =\frac{1}{j!}\left|\det
    \begin{pmatrix}
    1&\cdots&1\\
    p_0&\cdots&p_j
    \end{pmatrix}
    \right|.
\]
We set
\[
v_0(T)=1,\qquad v_j(T)=\Vol_j(C_j(T))\quad(1\le j\le n).
\]
Liu~\cite[Theorem~1.2]{liu2005ehrhart} proved that 
\begin{equation}\label{eq:liu}
    E_{C_d(T)}(m)=\sum_{j=0}^d v_j(T)m^j.
\end{equation}
If $a_i(T,d)$ denotes the $i$-th magic coefficient, then the change of basis gives
\begin{equation}\label{eq:inverse-basis}
    a_i(T,d)=\sum_{j=0}^i(-1)^{i-j}\binom{d-j}{i-j}v_j(T).
\end{equation}
Formula \eqref{eq:inverse-basis} contains alternating signs and therefore does not directly determine the signs of the magic coefficients. Instead, we prove an averaging identity involving Ehrhart polynomials of lattice zonotopes.

Write $[n]=\{1,\ldots,n\}$ and define
\[
g_i=t_i-t_{i-1},\qquad I_i=\{t_{i-1},t_{i-1}+1,\ldots,t_i-1\}\quad(i\in[n]).
\]
Let
\[
\Sc=\prod_{i=1}^n I_i, \qquad G=|\Sc|=\prod_{i=1}^n g_i.
\]
Recall that a lattice zonotope in $\RR^d$ is a Minkowski sum of line segments $[0,a_1]+\cdots+[0,a_m]$ with $a_1,\ldots,a_m\in\ZZ^d$. For $\mathbf s=(s_1,\ldots,s_n)\in\Sc$, set 
\begin{equation}\label{eq:zonotope}
    Z_d(\mathbf s)=\sum_{i=1}^n[0,g_i u_d(s_i)],\qquad u_d(s)=(1,s,\ldots,s^{d-1}).
\end{equation}
Since $s_1<\cdots<s_n$, every set of $d$ generators is linearly independent by the Vandermonde determinant. Thus $Z_d(\mathbf s)$ is a $d$-dimensional lattice zonotope.

Our main result is the following averaging identity.

\begin{thm}\label{thm:average}
For every $T=\{t_0<\cdots<t_n\}\subset\ZZ$ and $1\le d\le n$,
\begin{equation}\label{eq:average}
    E_{C_d(T)}(m)=\frac{1}{G}\sum_{\mathbf s\in\Sc}E_{Z_d(\mathbf s)}(m).
\end{equation}
\end{thm}

Theorem~\ref{thm:average}, together with the magic positivity of lattice zonotopes~\cite{beck2019h}, yields the following corollary.

\begin{cor}\label{cor:magic}
For every finite $T\subset\ZZ$ and $1\le d<|T|$, all magic coefficients $a_i(T,d)$ are nonnegative integers. Moreover, $h^\ast_{C_d(T)}(z)$ is real-rooted, and its coefficient sequence is log-concave and unimodal.
\end{cor}

\section{Proof of the averaging identity}\label{sec:average-proof}

\subsection{Volumes and finite differences}

We use the following volume formula of Nudel'man \cite{nudel1975isoperimetric}.

\begin{lem}[{\cite[\S3, Theorem, pp.~282--283]{nudel1975isoperimetric}}]\label{lem:edge-volume}
Let $x_0,\ldots,x_n\in\RR^j$, where $1 \le j \le n$, satisfy
\[
\det
\begin{pmatrix}
1 & \cdots & 1\\
x_{i_0} & \cdots & x_{i_j}
\end{pmatrix}
>0
\qquad
(0\le i_0<\cdots<i_j\le n).
\]
If $e_i=x_i-x_{i-1}$, then
\[
j!\Vol_j(\conv\{x_0,\ldots,x_n\})
=
\sum_{1\le i_1<\cdots<i_j\le n}
\det(e_{i_1},\ldots,e_{i_j}).
\]
\end{lem}

For $s_1<\cdots<s_j$, write
\[
\Delta(s_1,\ldots,s_j)=\prod_{1\le a<b\le j}(s_b-s_a).
\]
We set $\Delta(\emptyset)=1$; for $j=1$, the defining product is empty and therefore equals $1$.

\begin{lem}\label{lem:discrete-volume}
For $0\le j\le n$,
\begin{equation}\label{eq:discrete-volume}
    v_j(T)=\sum_{\substack{J\subseteq[n]\\|J|=j}} \ \sum_{(s_i)_{i\in J}\in\prod_{i\in J}I_i} \Delta(s_i:i\in J),
\end{equation}
where the arguments of $\Delta$ are listed in increasing index order.
\end{lem}

\begin{proof}
For $j=0$, the identity is $v_0(T)=1$. Assume $j\ge1$. By the binomial theorem,
\[
\nu_j(s+1)-\nu_j(s)=M_j u_j(s),\qquad 
(M_j)_{k\ell}=
\begin{cases}
    \binom{k}{\ell-1},&\ell\le k,\\
    0,&\ell>k,
\end{cases}
\qquad 1\le k,\ell\le j.
\]
The matrix $M_j$ is lower triangular with diagonal entries $1,2,\ldots,j$, so
\[
\det M_j=j!.
\]
Since $I_i=\{t_{i-1},\ldots,t_i-1\}$, summing these consecutive differences gives
\begin{equation}\label{eq:edge-sum}
    \nu_j(t_i)-\nu_j(t_{i-1}) 
    =\sum_{s\in I_i}\bigl(\nu_j(s+1)-\nu_j(s)\bigr)
    =M_j\sum_{s\in I_i}u_j(s).
\end{equation}

For $0\le i_0<\cdots<i_j\le n$, we have
\[
\det\begin{pmatrix}
    1&\cdots&1\\
    \nu_j(t_{i_0})&\cdots&\nu_j(t_{i_j})
\end{pmatrix}
=\prod_{0\le a<b\le j}(t_{i_b}-t_{i_a})>0.
\]
Thus Lemma~\ref{lem:edge-volume} applies. If $1\le i_1<\cdots<i_j\le n$, then~\eqref{eq:edge-sum} and multilinearity of the determinant give
\begin{align*}
 &\det\bigl(\nu_j(t_{i_1})-\nu_j(t_{i_1-1}),\ldots,
             \nu_j(t_{i_j})-\nu_j(t_{i_j-1})\bigr)\\
 &\qquad=\det(M_j)
   \sum_{\substack{s_1\in I_{i_1},\ldots,s_j\in I_{i_j}}}
   \det\bigl(u_j(s_1),\ldots,u_j(s_j)\bigr)\\
 &\qquad=j!
   \sum_{\substack{s_1\in I_{i_1},\ldots,s_j\in I_{i_j}}}
   \Delta(s_1,\ldots,s_j).
\end{align*}
Here $s_1<\cdots<s_j$ follows from $i_1<\cdots<i_j$ and the definition of the intervals $I_i$. Substituting this expression into
Lemma~\ref{lem:edge-volume} and cancelling the common factor $j!$ gives~\eqref{eq:discrete-volume}.
\end{proof}

\subsection{Arithmetic multiplicities}

For a list of linearly independent vectors $(a_i)_{i\in J}$,
define
\[
\mu(J)=\bigl[\operatorname{span}_{\RR}\{a_i:i\in J\}\cap\ZZ^d:\sum_{i\in J}\ZZ a_i\bigr].
\]
Equivalently, $\mu(J)$ is the greatest common divisor of the absolute values of the nonzero $|J|\times|J|$ minors of the matrix with columns $(a_i)_{i\in J}$. Set $\mu(\emptyset)=1$. Stanley's formula~\cite[Theorem~2.2]{stanley1991zonotope} states that
\begin{equation}\label{eq:stanley}
    E_{\sum_{i=1}^n[0,a_i]}(m)
    =\sum_{\substack{J\subseteq[n]\\(a_i)_{i\in J}\text{ independent}}} \mu(J)m^{|J|}.
\end{equation}
We use the standard fact that every alternating polynomial in $\ZZ[X_1,\ldots,X_j]$ is divisible by the Vandermonde polynomial
\[
\prod_{1\le a<b\le j}(X_b-X_a).
\]

\begin{lem}\label{lem:zonotope-coefficients}
For every $\mathbf s\in\Sc$,
\begin{equation}\label{eq:zonotope-coefficients}
    E_{Z_d(\mathbf s)}(m)= \sum_{\substack{J\subseteq[n]\\|J|\le d}} \left(\prod_{i\in J}g_i\right) \Delta(s_i:i\in J)m^{|J|}.
\end{equation}
\end{lem}

\begin{proof}
Fix $J=\{i_1<\cdots<i_j\}\subseteq[n]$ with $1\le j\le d$. First consider the $d\times j$ matrix with columns
\[
u_d(s_{i_1}),\ldots,u_d(s_{i_j}).
\]
Its minor in the first $j$ rows is the Vandermonde determinant
\begin{equation}\label{eq:first-minor}
    \Delta(s_{i_1},\ldots,s_{i_j})>0.
\end{equation}

For any choice of $j$ rows, say $1\le r_1<\cdots<r_j\le d$, the corresponding minor is obtained by specializing
\[
f(X_1,\ldots,X_j)=\det\bigl((X_b^{r_a-1})_{1\le a,b\le j}\bigr)
\]
at $X_b = s_{i_b}$. This polynomial is alternating with integer coefficients. Therefore the Vandermonde polynomial
\[
\Delta(X_1,\ldots,X_j)=\prod_{a<b}(X_b-X_a)
\]
divides $f$ in $\ZZ[X_1,\ldots,X_j]$. After substituting $X_b=s_{i_b}$, every $j\times j$ minor is therefore an integer multiple of~\eqref{eq:first-minor}. Since~\eqref{eq:first-minor} itself is one of the minors, the greatest common divisor of all maximal minors is exactly
\[
\Delta(s_{i_1},\ldots,s_{i_j}).
\]

Restoring the factor $g_i$ in each column multiplies every maximal minor by $\prod_{i\in J}g_i$. Consequently,
\[
\mu(J)=\left(\prod_{i\in J}g_i\right)\Delta(s_i:i\in J).
\]
The nonzero minor \eqref{eq:first-minor} shows that every set of at most $d$ generators is linearly independent, while any set of more than $d$ generators is dependent since the generators lie in $\RR^d$. Together with $\mu(\emptyset)=\Delta(\emptyset)=1$, substituting the above multiplicities into \eqref{eq:stanley} gives \eqref{eq:zonotope-coefficients}.
\end{proof}

\subsection{Averaging the coefficients}

\begin{proof}[Proof of Theorem~\ref{thm:average}]
Fix $0\le j\le d$. For $J\subseteq[n]$ with $|J|=j$, write
\[
g_J=\prod_{i\in J}g_i,
\qquad
\Sc_J=\prod_{i\in J}I_i,
\]
and for $\mathbf s_J=(s_i)_{i\in J}\in\Sc_J$ write $\Delta(\mathbf s_J)=\Delta(s_i:i\in J)$, with the arguments ordered by the indices in $J$. For $J=\emptyset$, we use the usual conventions and let $\Sc_{\emptyset}$ consist of one empty tuple.

By Lemma~\ref{lem:zonotope-coefficients}, for each fixed $\mathbf s\in\Sc$, the summand indexed by $J$ in the coefficient $[m^j]E_{Z_d(\mathbf s)}(m)$ is $g_J\Delta(\mathbf s_J)$. It depends only on the coordinates indexed by $J$. For every fixed $\mathbf s_J\in\Sc_J$, the remaining coordinates $(s_i)_{i\notin J}$ can be chosen in exactly
\[
\prod_{i\notin J}g_i=\frac{G}{g_J}
\]
ways. Hence
\begin{align*}
 \frac1G\sum_{\mathbf s\in\Sc}
       g_J\Delta(\mathbf s_J)
 &=\frac1G\frac{G}{g_J}g_J
       \sum_{\mathbf s_J\in\Sc_J}\Delta(\mathbf s_J)\\
 &=\sum_{\mathbf s_J\in\Sc_J}\Delta(\mathbf s_J).
\end{align*}
Summing over all $J\subseteq[n]$ of size $j$ and applying Lemma~\ref{lem:discrete-volume}, we obtain
\[
[m^j]\left(\frac1G\sum_{\mathbf s\in\Sc}E_{Z_d(\mathbf s)}(m)\right)=v_j(T).
\]
By Liu's formula~\eqref{eq:liu}, this is also the coefficient of $m^j$ in $E_{C_d(T)}(m)$. Since this holds for every $0\le j\le d$, identity~\eqref{eq:average} follows.
\end{proof}

\begin{proof}[Proof of Corollary~\ref{cor:magic}]
Taking magic coefficients in the identity of Theorem~\ref{thm:average}, the magic positivity of lattice zonotopes implies that $E_{C_d(T)}(m)$ is magic positive. On the other hand, Lemma~\ref{lem:discrete-volume} expresses each $v_j(T)$ as a finite sum of integer Vandermonde products, and therefore $v_j(T)\in\ZZ$ for every $j$. Hence, by~\eqref{eq:inverse-basis},
\[
a_i(T,d) = \sum_{j=0}^i(-1)^{i-j} \binom{d-j}{i-j}v_j(T)\in\ZZ.
\]
Together with magic positivity, this gives $a_i(T,d)\in\ZZ_{\ge0}$. Real-rootedness of $h^\ast_{C_d(T)}(z)$ follows from \cite[Theorem~4.19]{ferroni2024examples}. 

Since $h^\ast_{C_d(T)}(z)$ has nonnegative coefficients, its real-rootedness implies log-concavity by Newton's inequalities~\cite{stanley1989log}, and hence unimodality.
\end{proof}

\begin{ex}\label{ex:gaps}
Let $d=2$ and $T=\{0,1,3\}$. Then $g_1=1$, $g_2=2$, $I_1=\{0\}$, and $I_2=\{1,2\}$. The two zonotopes are
\[
Z_2(0,1)=[0,(1,0)]+[0,(2,2)],\qquad Z_2(0,2)=[0,(1,0)]+[0,(2,4)].
\]
Their Ehrhart polynomials are $1+3m+2m^2$ and $1+3m+4m^2$, respectively. Their magic coefficients are $(1,1,0)$ and $(1,1,2)$, respectively, so Theorem~\ref{thm:average} gives
\[
E_{C_2(T)}(m)=1+3m+3m^2=(1+m)^2+m(1+m)+m^2.
\]
In this case,
\[
h^\ast_{C_2(T)}(z)=1+4z+z^2.
\]
\end{ex}

\subsection*{Declaration on the use of generative AI}
The author used OpenAI's ChatGPT for exploratory discussions of proof strategies, English-language editing, and \LaTeX{} assistance. All mathematical arguments and references were verified by the author.

\subsection*{Acknowledgements}
The author would like to thank Akihiro Higashitani for carefully reading an earlier version of this manuscript and for valuable comments and suggestions that improved the presentation.

\bibliographystyle{plain} 
\bibliography{ref}

@article{branden2006linear,
  title={On linear transformations preserving the {P{\'o}lya} frequency property},
  author={Br{\"a}nd{\'e}n, Petter},
  journal={Transactions of the American Mathematical Society},
  volume={358},
  number={8},
  pages={3697--3716},
  year={2006}
}

@article{beck2019h,
  title={{$h^\ast$}-polynomials of zonotopes},
  author={Beck, Matthias and Jochemko, Katharina and McCullough, Emily},
  journal={Transactions of the American Mathematical Society},
  volume={371},
  number={3},
  pages={2021--2042},
  year={2019}
}

@article{ferroni2024examples,
  TITLE={Examples and counterexamples in {E}hrhart theory},
  AUTHOR={Ferroni, Luis and Higashitani, Akihiro},
  JOURNAL={EMS Surveys in Mathematical Sciences},
  YEAR={2024}
}

@article{avila2026luck,
  title={Luck and magic for {Pitman--Stanley} polytopes and parking functions},
  author={Avila, Nicolas and Ferroni, Luis and Morales, Alejandro H.},
  journal={arXiv preprint arXiv:2603.19194},
  year={2026}
}

@article{athanasiadis2026lattice,
  title={Lattice point enumeration of some arbor polytopes},
  author={Athanasiadis, Christos A. and Xiao, Qiqi and Yan, Xue},
  journal={arXiv preprint arXiv:2603.11654},
  year={2026}
}

@article{hill2026lattice,
  title={Lattice slices, {Ehrhart} polynomials, and magic positivity of generalized parking-function polytopes},
  author={Hill, Charlie and Luo, Ambrose and Trinh, Vu and Vindas-Mel{\'e}ndez, Andr{\'e}s R.},
  journal={arXiv preprint arXiv:2607.15503},
  year={2026}
}

@article{konoike2024new,
  title={A new class of magic positive {Ehrhart} polynomials of reflexive polytopes},
  author={Konoike, Masato},
  journal={arXiv preprint arXiv:2409.16648},
  year={2024}
}

@article{konoike2025magic,
  title={On the magic positivity of {Ehrhart} polynomials of dilated polytopes},
  author={Konoike, Masato},
  journal={arXiv preprint arXiv:2504.21395},
  year={2025}
}

@article{liu2026magic,
  title={Magic Positivity for the {Ehrhart} Polynomials of Partial Permutohedra},
  author={Liu, Feihu and Zhang, Zihao},
  journal={arXiv preprint arXiv:2607.03854},
  year={2026}
}

@article{liu2005ehrhart,
  title={Ehrhart polynomials of cyclic polytopes},
  author={Liu, Fu},
  journal={Journal of Combinatorial Theory, Series A},
  volume={111},
  number={1},
  pages={111--127},
  year={2005},
  publisher={Elsevier}
}

@article{nudel1975isoperimetric,
  title={Isoperimetric problems for the convex hulls of polygonal lines and curves in multidimensional spaces},
  author={Nudel'man, Adolf Abramovich},
  journal={Mathematics of the USSR-Sbornik},
  volume={25},
  number={2},
  pages={276--294},
  year={1975}
}

@inproceedings{stanley1991zonotope,
  title={A Zonotope Associated with Graphical Degree Sequences.},
  author={Stanley, Richard P.},
  booktitle={Applied geometry and discrete mathematics},
  pages={555--570},
  year={1991}
}

@article{stanley1989log,
  title={Log-Concave and Unimodal Sequences in Algebra, Combinatorics, and Geometry},
  author={Stanley, Richard P.},
  journal={Annals of the {New York} Academy of Sciences},
  volume={576},
  number={1},
  pages={500--535},
  year={1989},
  publisher={Wiley Online Library}
}

\end{document}